\documentclass[reqno,12pt]{amsart}

\usepackage{amsmath, amsthm, amsfonts, amssymb}
\input{mathrsfs.sty}

\usepackage{amsmath}
\usepackage{amsfonts}
\usepackage{amssymb}
\usepackage{amsthm}
\usepackage{amstext}

   \theoremstyle{plain}
   \newtheorem{thm}{Theorem}[section]
   \newtheorem{prop}[thm]{Proposition}
   \newtheorem{lem}[thm]{Lemma}
   \newtheorem{cor}[thm]{Corollary}
   \theoremstyle{definition}
   
   \newtheorem{defn}[thm]{Definition}
   \newtheorem{example}[thm]{Example}
   \theoremstyle{remark}

 \numberwithin{equation}{section}

\author{V. Manuilov}

\date{}

\address{Moscow Center for Fundamental and Applied Mathematics {\rm and} Moscow State University,
Leninskie Gory 1, Moscow, 
119991, Russia}

\email{manuilov@mech.math.msu.su}

\title[Inverse semigroup from metrics on doubles III]{Inverse semigroup from metrics on doubles III. Commutativity and (in)finiteness of idempotents}

\begin{document}

\begin{abstract}
We have shown recently that, given a metric space $X$, the coarse equivalence classes of metrics on the two copies of $X$ form an inverse semigroup $M(X)$. Here we study the property of idempotents in $M(X)$ of being finite or infinite, which is similar to this property for projections in $C^*$-algebras. We show that if $X$ is a free group then the unit of $M(X)$ is infinite, while if $X$ is a free abelian group then it is finite. As a by-product, we show that the inverse semigroup $M(X)$ is not a quasi-isometry invariant. More examples of finite and infinite idempotents are provided. We also give a geometric description of spaces, for which their inverse semigroup $M(X)$ is commutative.

\end{abstract}

\maketitle

\section{Introduction}

Given metric spaces $(X,d_X)$ and $(Y,d_Y)$, a metric $d$ on $X\sqcup Y$ that extends the metrics $d_X$ on $X$ and $d_Y$ on $Y$, depends only on the values $d(x,y)$, $x\in X$, $y\in Y$, and it may be not easy to check which functions $d:X\times Y\to (0,\infty)$ determine a metric on $X\sqcup Y$. The problem of description of all such metrics is difficult due to the lack of a nice algebraic structure on the set of metrics, but, passing to coarse equivalence of metrics, we get an algebraic structure, namely, that of an inverse semigroup \cite{M}. Recall that two metrics, $b,d$, on a space $Z$ are coarsely equivalent if there exist monotone functions $\varphi,\psi:[0,\infty)\to[0,\infty)$ such that 
$$
\lim\nolimits_{t\to\infty}\varphi(t)=\lim\nolimits_{t\to\infty}\psi(t)=\infty
$$ 
and 
$$
\varphi(d(z_1,z_2))\leq b(z_1,z_2)\leq\psi(d(z_1,z_2))
$$ 
for any $z_1,z_2\in Z$. A coarse equivalence class of a metric $d$ we denote by $[d]$.

Let $\mathcal M(X,Y)$ denote the set of all metrics $d$ on $X\sqcup Y$ such that
\begin{itemize}
\item
the restriction of $d$ onto $X$ and $Y$ are $d_X$ and $d_Y$ respectively;
\item
$\inf_{x\in X,y\in Y}d(x,y)>0$. 
\end{itemize}

Coarse equivalence classes of metrics in $\mathcal M(X,Y)$ can be considered as morphisms from $X$ to $Y$ \cite{Manuilov-Morphisms}, where the composition $\rho d$ of a metric $d$ on $X\sqcup Y$ and a metric $\rho$ on $Y\sqcup Z$ is given by the metric determined by 
$$
(\rho\circ d)(x,z)=\inf\nolimits_{y\in Y}[d(x,y)+\rho(y,z)],\quad x\in X,\ z\in Z. 
$$
When $Y=X$, we call $X\sqcup X$ the double of $X$. In what follows we identify $X\sqcup X$ with $X\times\{0,1\}$, and write $X$ for $X\times\{0\}$ (resp., $x$ for $(x,0)$) and $X'$ for $X\times\{1\}$ (resp., $x'$ for $(x,1)$).

The main result of \cite{M} is that the semigroup $M(X)$ (with respect to this composition) of coarse equivalence classes of metrics on the double of $X$ is an inverse semigroup with the unit element ${\mathbf 1}$ and the zero element ${\mathbf 0}$, and the unique pseudo-inverse for $[d]\in M(X)$ is the coarse equivalence class of th metric $d^*$ given by $d^*(x,y')=d(x',y)$, $x,y\in X$.

Recall that a semigroup $S$ is an inverse semigroup if for any $s\in S$ there exists a unique $t\in S$ (denoted by $s^*$ and called a pseudo-inverse) such that $s=sts$ and $t=tst$ \cite{Lawson}. Philosophically, inverse semigroups describe local symmetries in a similar way as groups describe global symmetries, and technically, the construction of the (reduced) group $C^*$-algebra of a group generalizes to that of the (reduced) inverse semigroup $C^*$-algebra \cite{Paterson}. It is known that any two idempotents in an inverse semigroup $S$ commute, and that there is a partial order on $S$ defined by $s\leq t$ if $s=ss^*t$. Our standard reference for inverse semigroups is \cite{Lawson}.

Close relation between inverse semigroups and $C^*$-algebras allows to use classification of projections in $C^*$-algebras for idempotents in inverse semigroups. Namely, as in $C^*$-algebra theory, we call two idempotents, $e,f\in E(S)$ von Neumann equivalent (and write $e\sim f$) if there exists $s\in S$ such that $s^*s=e$, $ss^*=f$. An idempotent $e\in E(S)$ is called {\em infinite} if there exists $f\in E(S)$ such that $f\preceq e$, $f\neq e$, and $f\sim e$. Otherwise $e$ is {\em finite}. An inverse semigroup is {\em finite} if every idempotent is finite, and is {\em weakly finite} if it is unital and the unit is finite. A commutative inverse semigroup is patently finite.

In \cite{M2} we gave a geometric description of idempotents in the inverse semigroup $M(X)$ (there are two types of idempotents, named type I and type II) and showed in Lemma 3.3 that the type is invariant under the von Neumann equivalence.  
In the first part of this paper, we study the property of weak finiteness for $M(X)$ (i.e. finiteness of the unit element) and discuss its relation to geometric properties of $X$. 

We start with several examples of finite or infinite idempotents, and then show that if $X$ is a free group then $M(X)$ is not weakly finite, while if $X$ is a free abelian group then it is weakly finite. We also show that the inverse semigroup $M(X)$ is not a quasi-isometry invariant. The property of being weakly finite is also not a coarse invariant. We don't know if it is a quasi-isometry invariant.

In the second part of this paper, we give a geometric description of spaces, for which the inverse semigroup $M(X)$ is commutative.

\part{Weak finiteness of $M(X)$}

\section{Some examples}

The following example shows that in $M(X)$, for an appropriate $X$, we can imitate examples of partial isometries and projections in a Hilbert space.  

\begin{example}
Let $l^1(\mathbb N)$ be the space of infinite $l^1$ sequences, with the metric given by the $l^1$-norm, and let 
$$
X_n=\{(0,\ldots,0,t,0,\ldots):t\in[0,\infty)\}
$$ 
with $t$ at the $n$-th place, $n\in\mathbb N$. Set $X=\cup_{n\in\mathbb N}X_n\subset l^1(\mathbb N)$.

Let $x=(0,\ldots,0,t,0,\ldots)\in X_n$, $y=(0,\ldots,0,s,0,\ldots)\in X_m$.
Define metrics $d$, $e$, $f$ on $X\sqcup X'$ by 
\begin{eqnarray*}
d(x,y')&=&\left\lbrace\begin{array}{cl}|s-t|+1,&\mbox{if\ }m=n+1;\\
s+t+1,&\mbox{if\ }m\neq n+1,\end{array}\right.\\
e(x,y')&=&\left\lbrace\begin{array}{cl}|s-t|+1,&\mbox{if\ }m=n;\\
s+t+1,&\mbox{if\ }m\neq n,\end{array}\right.\\
f(x,y')&=&\left\lbrace\begin{array}{cl}|s-t|+1,&\mbox{if\ }m=n\geq 2;\\
s+t+1,&\mbox{if\ }m\neq n \mbox{\ or\ }m=n=1.\end{array}\right.
\end{eqnarray*}

It is easy to see that $d$, $e$, $f$ are metrics on $X\sqcup X'$, and that $d^*d=e$, $dd^*=f$ are idempotents, and that $e={\mathbf 1}\neq f$. In particular, ${\mathbf 1}$ is infinite. Although $d$ seems similar to a one-sided shift in a Hilbert space, it behaves differently: $f$ is orthogonally complemented, i.e. there exists $h$ such that $f\lor h={\mathbf 1}$, $f\land h={\mathbf 0}$, but the complement is not a minimal idempotent, i.e. there exists a lot of idempotents $j\in E(M(X))$ such that $j\leq h$, $j\neq h$. 

\end{example}

On the other hand, if $X\subset [0,\infty)$ with the standard metric then the inverse semigroup $M(X)$ is commutative (Prop. 7.1 in \cite{M}), hence any idempotent can be equivalent only to itself, hence is finite. In Part 2, we shall give a geometric description of all metric spaces with commutative $M(X)$, which is patently finite.


The next example shows that the picture may be more complicated.

\begin{thm}\label{example-amenable}
There exists an amenable space $X$ of bounded geometry and $s\in M(X)$ such that $s^*s=\mathbf 1$, but $ss^*\neq \mathbf 1$. 

\end{thm}
\begin{proof}
Let 
$$
x_n=(\log 2,\log 3,\ldots,\log n,\log (n+1),0,0,\ldots),
$$ 
$$
x'_n=(\log 2,\log 2,\log 3,\log 3,\ldots,\log n,\log n,\log (n+1),\log (n+1),0,0,\ldots),
$$
and let 
$$
X=\{x_n:n\in\mathbb N\}, \quad X'=\{x'_n:n\in\mathbb N\}, 
$$
$X,X'\subset l^\infty(\mathbb N)$ with the metric 
$$
d(x,y)=\sup\nolimits_{k}|x_k-y_k|,\quad x=(x_1,x_2,\ldots),\quad y=(y_1,y_2,\ldots). 
$$
Take $m>n$, then $d(x_n,x_m)=\log m=d(x'_n,x'_m)$, hence the restriction of $d$ onto the two copies of $X$ coincide (thus determining the metric $d_X$ on $X$), and $d\in\mathcal M(X)$. 

We have, for $n$ even, $n=2k$, 
\begin{eqnarray*}
d(x_n,X')&=&\inf\nolimits_{m\in\mathbb N}d(x_{2k},x'_m)=d(x_{2k},x'_k)\\
&=&\max\nolimits_{i\leq k}(|\log (i+1)-\log(2i+1)|\ \leq\  \log 2,
\end{eqnarray*}
and for $n$ odd, $n=2k-1$,
$$
d(x_{2k-1},x'_m)\geq \log(k+1)
$$
for any $m\in\mathbb N$, hence
$$
d(x_n,X')=\inf\nolimits_{m\in\mathbb N}d(x_{2k-1},x'_m)= d(x_{2k-1},x'_{k})=\log(k+1),
$$
i.e.
$$
\lim_{k\to\infty}d(x_{2k-1},X')=\infty.
$$
On the other hand,
\begin{eqnarray*}
d(x'_n,X)&=&\inf\nolimits_{m\in\mathbb N}d(x'_n,x_m)\leq d(x'_n,x_{2n})\\
&=&\log(2n+1)-\log(n+1)\ \leq\ \log 2.
\end{eqnarray*}

Let $X_+=\{x_{2k}:k\in\mathbb N\}$, $X_-=\{x_{2k-1}:k\in\mathbb N\}$.
Then  
\begin{eqnarray*}
d^*d(x,x')&=&\inf\nolimits_{y\in X}[d(x,y')+d^*(y,x')]=\inf\nolimits_{y\in X}[d(x,y')+d(x,y')]\\
&=&2d(x,X')\ \leq\ \log 2
\end{eqnarray*}
for any $x\in X_+$ and
$$
\lim_{x\in X_-;\ x\to\infty}d^*d(x,x')=\lim_{x\in X_-;\ x\to\infty}2d(x,X')=\infty,
$$
while
$$
dd^*(x,x')=2d(x',X)\leq\log 2
$$
for any $x\in X$.

Let $d_+, d_-\in\mathcal M(X)$ be the idempotent selfadjoint metrics defined by
$$
d_\pm(x,y')=\inf\nolimits_{u\in X_\pm}[d_X(x,u)+1+d_X(u,y)].
$$

Then $[d^*d]=[d_+]$, $[dd^*]=\mathbf 1$, and $[d_+]$ is strictly smaller than $\mathbf 1$, hence $M(X)$ is not finite.

Note that $X$ is amenable. Set $F_n=\{x_1,\ldots,x_n\}\subset X$. Let $N_r(A)$ denote the $r$-neighborhood of the set $A$. Then $N_r(F_n)\setminus F_n$ is empty when $\log(n+1)>r$, hence $\{F_n\}_{n\in\mathbb N}$ is a F\o lner sequence. 
For $r=\log m$, the ball $B_r(x_n)$ of radius $r$ centered at $x_n$ contains either no other points besides $x_n$ (if $n\geq m+1$), or it consists of the points $x_1,\ldots,x_m$ (if $n\leq m$), hence the metric on $X$ is of bounded geometry. In fact, this space is of asymptotic dimension zero.  

\end{proof}

\section{Case of free groups}

Let $X=\Gamma$ be a finitely generated group with the word length metric $d_X$. Consider the following property (I):

\begin{itemize}
\item[(i1)]
$X=Y\sqcup Z$, and for any $D>0$ there exists $z\in Z$ such that $d_X(z,Y)>D$;
\item[(i2)]
there exist $g,h\in\Gamma$ such that $gY\subset Y$, $hZ\subset Y$ and $gY\cap hZ=\emptyset$;
\item[(i3)]
there exists $C>0$ such that $|d_X(gy,hz)-d_X(y,z)|<C$ for any $y\in Y$, $z\in Z$.

\end{itemize}

The property (I) is neither stronger nor weaker than non-amenability. If we require additionally that $Y\sim Z$ then it would imply non-amenability.

\begin{lem}
The free group $\mathbb F_2$ on two generators satisfies the property (I).

\end{lem}
\begin{proof}
Let $a$ and $b$ be the generating elements of $\mathbb F_2$, and let $Y\subset X$ be the set of all reduced words in $a$, $a^{-1}$, $b$ and $b^{-1}$ that begin with $a$ or $a^{-1}$, $Z=X\setminus Y$. Let $g=ab$, $h=a^2$. Clearly, $gY\subset Y$ and $hZ\subset Y$. 

If $z$ begins with $a^n$, $n>D$, then $d_X(z,Y)\geq n$.

If $y\in Y$, $z\in Z$ then 
$$
d_X(aby,a^2z)=|y^{-1}b^{-1}a^{-1}a^2z|=|y^{-1}b^{-1}az|=|y^{-1}z|+2=d_X(y,z)+2,
$$ 
as the word $y^{-1}b^{-1}az$ cannot be reduced any further ($y^{-1}$ ends with $a^{\pm}$, and $z$ either begins with $b^{\pm}$, or is an empty word).

\end{proof}

\begin{thm}
Let $X=\Gamma$ be a group with the property (I). Then $X$ is not weakly finite. 

\end{thm}
\begin{proof}
We shall prove that there exists $d\in\mathcal M(X)$ such that $[d^*d]=\mathbf 1$ and $[dd^*]\neq\mathbf 1$.

Let $X=Y\sqcup Z$, $g,h\in\Gamma$ satisfy the conditions of the property (I).
Define a map $f:X\to X$ by setting 
$$
f(x)=\left\lbrace \begin{array}{cl}gx,&\mbox{if\ }x\in Y;\\
hx,&\mbox{if\ }x\in Z.
\end{array}\right.
$$ 

The maps $f|_Y$ and $f|_Z$ are left multiplications by $g$ and $h$, respectively, hence are isometries. If $y\in Y$, $z\in Z$ then (i3) holds for some $C>0$, hence $|d_X(f(x),f(y))-d_X(x,y)|<C$ holds for any $x,y\in X$.

Set 
$$
d(x,y')=\inf\nolimits_{u\in X}[d_X(x,u)+C+d_X(f(u),y)]. 
$$
It is easy to check that $d$ satisfies all triangle inequalities, hence $d$ is a metric, $d\in\mathcal M(X)$. Then 
$$
d^*d(x,x')=2d_X(x,X')=2\inf\nolimits_{u,y\in X}[d_X(x,u)+C+d_X(f(u),y)]=2C
$$
for any $x\in X$, and
\begin{eqnarray*}
dd^*(x,x')&=&2d_X(x',X)=\inf\nolimits_{z,u\in X}[d_X(z,u)+C+D_X(f(u),x)]\\
&=&2(C+d_X(f(X),x))\ \geq\ 2C+2d_X(Y,x)
\end{eqnarray*}
is not bounded. Thus, $[d^*d]=\mathbf 1$, while $[dd^*]\neq\mathbf 1$.

\end{proof}

\section{Case of abelian groups}

A positive result is given by the following Theorem.
\begin{thm}
Let $X=\mathbb R^n$, with a norm $\|\cdot\|$, and let the metric $d_X$ be determined by the norm $\|\cdot\|$. If $s\in M(X)$, $s^*s=\mathbf 1$ then $ss^*=\mathbf 1$. 

\end{thm}
\begin{proof}
Let $d\in\mathcal M(X)$, $[d]=s$. As $[d^*d]=\mathbf 1$, there exists $C>0$ such that $2d(x,X')<C$ for any $x\in X$. It suffices to show that there exists $D>0$ such that $2d(x',X)<D$ for any $x\in X$. Suppose the contrary: for any $n\in\mathbb N$ there exists $x_n\in X$ such that $d(x'_n,X)>2n$. Then $d(y',X)>n$ for any $y\in X$ such that $d_X(y,x_n)\leq n$.

As $2d(x,X')<C$ for any $x\in X$, there is a (not continuous) map $f:X\to X$ such that $d_X(x,f(x)')<C/2$ for any $x\in X$. This map satisfies 
$$
|d_X(f(x),f(y))-d_X(x,y)|<C \quad\mbox{for\ any\ } x,y\in X. 
$$
Then there exists a continuous map $g:X\to X$ such that $d_X(f(x),g(x))<C$ for any $x\in X$. This map satisfies 
$$
|d_X(g(x),g(y))-d_X(x,y)|<2C\quad\mbox{for\ any\ } x,y\in X, 
$$
and $d(x,g(x)')<3C/2$ for any $x\in X$.

Let $x_0=0$ denote the origin of $X$, and let $S_R$ be the sphere of radius $R$ centered at the origin, 
$$
S_R=\{x\in X:d_X(x_0,x)=R\}.
$$ 
Set $x'_0=f(x_0)'\in X'$. For $x\in S_R$, we have 
$$
|d_X(x_0,g(x))|=|d_X(x'_0,g(x)')-R|\leq 2C. 
$$
For $R>3C$, set $h_R(x)=\frac{g(x)}{\|g(x)\|}$. Then $h_R$ is a coninuous map from $S_R$ to $S_R$ for any $R>3C$, and we have 
$$
|d_X(h_R(x),h_R(y))-d_X(x,y)|<3C\quad\mbox{for\ any\ } x,y\in S_R, 
$$
and $d(x,h_R(x)')<5C$ for any $x\in X$.

Let $R_n=d_X(x'_0,x'_n)$. Then $x_n\in S_{R_n}$. If $d_X(y,x_n)\leq n$ then $d(y',X)>n$, hence $y\notin h_{R_n}(S_{R_n})$, so the map $h_{R_n}$ is not surjective. Then, by the Borsuk--Ulam Theorem, there exists a pair of antipodal points $y_1,y_2\in S_{R_n}$ such that $h_{R_n}(y_1)=h_{R_n}(y_2)=z$. As $d(y_i,z')<5C$, $i=1,2$, and $d_X(y_1,y_2)=2R_n$, the triangle inequality (for the trianlge with the vertices $y_1$, $y_2$, $z'$) is violated when $2R_n>10C$. This contradiction proves the claim.  

\end{proof}

\begin{cor}
Let $X=\mathbb Z^n$ with an $l_p$-metric, $1\leq p\leq\infty$, and let $s\in M(X)$. Then $s^*s=\mathbf 1$ implies $ss^*=\mathbf 1$.

\end{cor}
\begin{proof}
By Proposition 9.2 of \cite{M}, $M(\mathbb Z^n)=M(\mathbb R^n)$.

\end{proof}

\section{$M(X)$ doesn't respect equivalences}

\begin{prop}
The inverse semigroup $M(X)$ is not a coarse invariant.

\end{prop}
\begin{proof}
The space $X$ from Theorem \ref{example-amenable} is coarsely equivalent to the space $Y=\{n^2:n\in\mathbb N\}$ with the standard metric, which we denote by $b_X$. Indeed, for $n<m$, we have $b_X(x_n,x_m)=m^2-n^2$ and $d_X(x_n,x_m)=\log(m+1)$. As $m^2-(m-1)^2=2m-1>\log(m+1)$ for $m>1$, we have $d_X(x,y)\leq b_X(x,y)$ for any $x,y\in X$, and taking $f(t)=2e^t$, we have $b_X(x,y)\leq f(d_X(x,y))$ for any $x,y\in X$.  

For the metric $d_X$ from Theorem \ref{example-amenable}, the inverse semigroup $M(X,d_X)$ is not commutative ($[d^*d]\neq [dd^*]$), while the inverse semigroup $M(X,b_X)$ is commutative by Prop. 7.1 of \cite{M2}.

\end{proof}

\begin{thm}
The inverse semigroup $M(X)$ is not a quasi-isometry invariant.

\end{thm}
\begin{proof}
Let $X=\mathbb N$ be endowed with the metric $b_X$ given by $b_X(n,m)=|2^n-2^m|$, $n,m\in\mathbb N$, and let $y_n=s(n)4^{[\frac{n}{2}]}$, where $s(n)=(-1)^{[\frac{n-1}{2}]}$ and $[t]$ is the greatest integer not exceeding $t$. Let $d_X$ be the metric on $X$ given by $d_X(n,m)=|y_n-y_m|$, $n,m\in\mathbb N$. The two metrics are quasi-isometric. Indeed, suppose that $n>m$. 
If $s(n)=-s(m)$ then
$$
d_X(n,m)=4^{[\frac{n}{2}]}+4^{[\frac{m}{2}]}\leq 4^{\frac{n}{2}+1}+4^{\frac{m}{2}+1}=4(\cdot 2^n+2^m)\leq 12 b_X(n,m);
$$
$$
d_X(n,m)=4^{[\frac{n}{2}]}+4^{[\frac{m}{2}]}\geq 4^{\frac{n}{2}}+4^{\frac{m}{2}}\geq 2^n- 2^m= b_X(n,m).
$$   
We use here that $\frac{2^r+1}{2^r-1}\leq 3$ for any $r=n-m\in\mathbb N$. 
If $s(n)=s(m)$ then  
$$
d_X(n,m)=4^{[\frac{n}{2}]}-4^{[\frac{m}{2}]}\leq 4^{\frac{n}{2}+1}-4^{\frac{m}{2}}=4\cdot 2^n-2^m\leq 7 b_X(n,m).
$$
We use here that $\frac{4\cdot 2^r-1}{2^r-1}\leq 7$ for any $r=n-m\in\mathbb N$.
To obtain an estimate in other direction, note that $s(n)=s(m)$ implies that $[n/2]\geq [m/2]+1$, and that $n-m\neq 2$. 
If $n=m+1$ then 
$$
d_X(m+1,m)=3\cdot 4^{[m/2]}\geq \frac{3}{2}\cdot 2^m=\frac{3}{2}\ b_X(m+1,m), 
$$
If $n\geq m+3$ then
$$
d_X(n,m)=4^{[\frac{n}{2}]}-4^{[\frac{m}{2}]}\geq 4^{\frac{n}{2}}-4^{\frac{m}{2}+1}= 2^n-4\cdot 2^m\geq \frac{4}{7}\ b_X(n,m).
$$   
We use here that $\frac{2^r-4}{2^r-1}\geq \frac{4}{7}$ for any $r=n-m\geq 3$.
Thus, 
$$
\frac{3}{7}\ b_X(n,m)\leq d_X(n,m)\leq 12\cdot b_X(n,m)
$$ 
for any $n,m\in\mathbb N$, so the two metrics are quasi-isometric.

We already know that $M(X,b_X)$ is commutative, so it remains to expose two non-commuting elements in $M(X,d_X)$.

Let 
$$
X=\{(y_n,0):n\in\mathbb N\},\quad X'=\{(-y_n,1):n\in\mathbb N\}, 
$$
and let $d$ be the metric on $X\sqcup X'$ induced from the standard metric on the plane $\mathbb R^2$, $s=[d]$. Note that $-y_n=y_{n-1}$ if $y_n>0$ and $n>1$, and $-y_n=y_{n+1}$ if $y_n<0$. Hence, $d^*=d$ and $s^2=\mathbf 1$. 

Let 
$$
A_+=\{y_n:n\in\mathbb N; y_n>0\},\quad A_-=\{y_n:n\in\mathbb N; y_n<0\}, 
$$
$X=A_+\sqcup A_-$, and let the metrics $d_+$ and $d_-$ on $X\sqcup X'$ be given by
$$
d_\pm(n,m')=\inf\nolimits_{k\in A_\pm}[d_X(n,k)+1+d_X(k,m)],
$$
$e=[d_+]$, $f=[d_-]$. Then $es=\mathbf 0$, while $se=f$, i.e. $e$ and $s$ do not commute. 

\end{proof}

\part{When $M(X)$ is commutative}

\section{R-spaces}

\begin{defn}
A metric space $X$ is an {\it R-space} (R for rigid) if, for any $C>0$ and any two sequences $\{x_n\}_{n\in\mathbb N}$, $\{y_n\}_{n\in\mathbb N}$ of points in $X$ satisfying 
\begin{equation}\label{near_nm}
|d_X(x_n,x_m)-d_X(y_n,y_m)|<C\quad\mbox{for\ any\ }n,m\in\mathbb N
\end{equation}
there exists $D>0$ such that $d_X(x_n,y_n)<D$ for any $n\in\mathbb N$. 

\end{defn}

\begin{example}
As $M(X)$ is commutative for any $X\subset[0,\infty)$, it would follow from Theorem \ref{R} below that such $X$ is an R-space. A less trivial example is a planar spiral $X$ given by $r=e^\varphi$ in polar coordinates with the metric induced from the standard metric on the plane. Indeed, take any two sequences $\{x_n\}_{n\in\mathbb N}$, $\{y_n\}_{n\in\mathbb N}$, in $X$. Without loss of generality we may assume that $x_1=y_1=0$ is the origin. If these sequences satisfy (\ref{near_nm}) then 
$$
|d_X(0,x_n)-d_X(0,y_n)|<C
$$ 
for some fixed $C>0$ (we take $m=1$). If $x_n=(r_n,\varphi_n)$, $y_n=(s_n,\psi_n)$ then  $d_X(0,x_n)=r_n$, $d_X(0,y_n)=s_n$, and we have $|r_n-s_n|<C$. Then $x_n$ and $y_n$ lie in a ring of width $C$, say $R\leq r\leq R+C$. If $R$ is sufficiently great then 
$$
d_X(x_n,y_n)\leq (\log(R+C)-\log R)(R+C), 
$$
which is bounded as a function of $R$.

\end{example}

Recall that two maps $f,g:X\to X$ are equivalent if there exists $C>0$ such that $d_X(f(x),g(x))<C$ for any $x\in X$. A map $f:X\to X$ is an almost isometry if there exists $C>0$ such that 
$$
|d_X(f(x),f(y))-d_X(x,y)|<C
$$ 
for any $x,y\in X$ and if for any $y\in X$ there exists $x\in X$ such that $d_X(f(x),y)<C$ (the latter condition provides existence of an `inverse' map $g:X\to X$ such that $f\circ g$ and $g\circ f$ are equivalent to the identity map). The set $AI(X)$ of all equivalence classes of almost isometries of $X$ is a group with respect to the composition. A metric space $X$ is called AI-rigid \cite{KLS} if the group $AI(X)$ is trivial.

\begin{prop}
A countable R-space $X$ is AI-rigid.

\end{prop}
\begin{proof}
Let $\{x_n\}_{n\in\mathbb N}$ be a sequence of all points of $X$, and let $f:X\to X$ be an almost isometry. Set $y_n=f(x_n)$. Then there exists $C>0$ such that 
$$
|d_X(f(x_n),f(x_m))-d_X(x_n,x_m)|<C
$$ 
for any $n,m\in\mathbb N$, hence there exists $D>0$ such that 
$$
d_X(x_n,f(x_n))=d_X(x_n,y_n)<D
$$ 
for any $n\in\mathbb N$, i.e. $f$ is equivalent to the identity map, hence $X$ is AI-rigid.  

\end{proof} 

\begin{example}
Euclidean spaces $\mathbb R^n$, $n\geq 1$, are not R-spaces, as they have a non-trivial symmetry. 
The Archimedean spiral $r=\varphi$ is not an R-space, as it is $\pi$-dense in $\mathbb R^2$.  

\end{example}

\section{Criterion of commutativity}

Let $a,b:T\to[0,\infty)$ be two functions on a set $T$. We say that $a\preceq b$ if there exists a monotonely increasing function $\varphi:[0,\infty)\to[0,\infty)$ with $\lim_{s\to\infty}\varphi(s)=\infty$ (we call such functions {\it reparametrizations}) such that $a(t)\leq\varphi(b(t))$ for any $t\in T$.

The following Lemma should be known, but we could not find a reference.
\begin{lem}
Let $a,b:T\to[0,\infty)$ be two functions. If $a\preceq b$ is not true then there exists $C>0$ and a sequence $(t_n)_{n\in\mathbb N}$ of points in $T$ such that $b(t_n)<C$ for any $n\in\mathbb N$ and $\lim_{n\to\infty}a(t_n)=\infty$. 

\end{lem}
\begin{proof}
If $a\preceq b$ is not true then for any reparametrization $\varphi$ there exists $t\in T$ such that $a(t)> \varphi(b(t))$. Suppose that for any $C>0$, the value $\max\{a(t):b(t)\leq C\}$ is finite. Then set 
$$
f(C)=\max(\max\{a(t):b(t)\leq C\},C). 
$$
This gives a reparametrization $f$. If $b(t)=C$ then $a(t)\leq f(C)=f(b(t))$ --- a contradiction. Thus, there exists $C>0$ such that $\max\{a(t):b(t)\leq C\}=\infty$. It remains to choose a sequence $(t_n)_{n\in\mathbb N}$ in the set $\{t\in T:b(t)\leq C\}$ such that $a(t_n)>n$.

\end{proof}

\begin{thm}\label{R}
$X$ is an R-space if and only if $M(X)$ is commutative.

\end{thm}
\begin{proof}
Let $X$ be an R-space. We shall show that any $s\in M(X)$ is a projection. It would follow that $M(X)$ is commutative.
First, we shall show that any $s\in M(X)$ is selfadjoint. Let $d\in\mathcal M(X)$, $[d]=s$. Suppose that $[d^*]\neq[d]$. This means that either $d^*\preceq d$ or $d\preceq d^*$ is not true, where $d$ and $d^*$ are considered as functions on $T=X\times X'$. Without loss of generality we may assume that $d^*\preceq d$ is not true. Then there exist sequences $(x_n)_{n\in\mathbb N}$ in $X$ and $(y')_{n\in\mathbb N}$ in $X'$ and $L>0$ such that $d(x_n,y'_n)<L$ for any $n\in\mathbb N$ and $\lim_{n\to\infty}d(y_n,x'_n)=\infty$ (recall that $d^*(x,y'):=d(y,x')$).

Take $n,m\in\mathbb N$.
Since $d(x_n,y'_n)<L$, $d(x_m,y'_m)<L$, we have 
$$
|d_X(x_n,x_m)-d_X(y_n,y_m)|=|d_X(x_n,x_m)-d_X(y'_n,y'_m)|<2L,
$$
and, since $X$ is an R-space, there exists $D>0$ such that $d_X(x_n,y_n)<D$ for any $n\in\mathbb N$. 

Then, using the triangle inequality for the quadrangle $x_ny_nx'_ny'_n$, we get
\begin{eqnarray*}
d(y_n,x'_n)&\leq& d_X(y_n,x_n)+d(x_n,y'_n)+d_X(y'_n,x'_n)\\
&=&d_X(y_n,x_n)+d(x_n,y'_n)+d_X(y_n,x_n)\\
&<&D+L+D,
\end{eqnarray*}
which contradicts the condition $\lim_{n\to\infty}d(y_n,x'_n)=\infty$.

Now, let us show that $[d]\in M(X)$ is idempotent if $X$ is an R-space. Let $a(x)=d(x,X')$, $b(x)=d(x,x')$. It was shown in \cite{M} (Theorem 3.1 and remark at the end of Section 11) that if $[d]$ is selfadjoint then it is idempotent if and only if 
$b\preceq a$. Suppose that the latter is not true. Then there exists $L>0$ and a sequence $\{x_n\}_{n\in\mathbb N}$ of points in $X$ such that $d(x_n,X')<L$ for any $n\in\mathbb N$ and $\lim_{n\to\infty}d(x_n,x'_n)=\infty$. In particular, this means that there exists a sequence $\{y_n\}_{n\in\mathbb N}$ of points in $X$ such that $d(x_n,y'_n)<L$ for any $n\in\mathbb N$. Since $[d]$ is selfadjoint, for any $L>0$ there exists $R>0$ such that if $d(x,y')<L$ then $d(x',y)<R$. 

It follows from the triangle inequality for the quadrangle $x_nx_my'_ny'_m$ that 
\begin{eqnarray*}
|d_X(x_n,x_m)-d_X(y_n,y_m)|&=&|d_X(x_n,x_m)-d_X(y'_n,y'_m)|\\&\leq& d(x_n,y'_n)+d(x_m,y'_m)\ <\ 2L
\end{eqnarray*}
for any $n,m\in\mathbb N$,
hence, the property of being an R-space implies that there exists $D>0$ such that $d_X(x_n,y_n)<D$ for any $n\in\mathbb N$. Therefore,
$$
d(x_n,x'_n)\leq d_X(x_n,y_n)+d(y_n,x'_n)<D+R
$$ 
for any $n\in\mathbb N$ --- a contradiction with $\lim_{n\to\infty}d(x_n,x'_n)=\infty$.

In the opposite direction, suppose that $X$ is not an R-space. i.e. that there exists $C>0$ and sequences $\{x_n\}_{n\in\mathbb N}$, $\{y_n\}_{n\in\mathbb N}$ of points in $X$ such that (\ref{near_nm}) holds and $\lim_{n\to\infty}d_X(x_n,y_n)=\infty$. 

Note that these sequences cannot be bounded. Indeed, if there exists $R>0$ such that $d_X(x_1,x_n)<R$ for any $n\in\mathbb N$ then 
$$
d_X(y_1,y_n)\leq d_X(x_1,x_n)+C=R+C
$$ 
for any $n\in\mathbb N$, but then 
$$
d_X(x_n,y_n)\leq d_X(x_n,x_1)+d_X(x_1,y_1)+d_X(y_1,y_n)< R+d_X(x_1,y_1)+R+C, 
$$
which contradicts $\lim_{n\to\infty}d_X(x_n,y_n)=\infty$. Passing to a subsequence, we may assume that 
$$
d_X(x_k,x_n)>k,\ d_X(x_k,y_n)>k,\ d_X(y_k,x_n)>k,\ d_X(y_k,y_n)>k
$$ for any $n<k$, and $d_X(x_k,y_k)>k$ for any $k\in\mathbb N$. In particular, this means that
\begin{equation}\label{knm}
d_X(x_k,y_n)>k\quad\mbox{for\ any\ }k,n\in\mathbb N.
\end{equation}
 
Let us define two metrics on $X$ and show that they don't commute.
For $x,y\in X$ set 
$$
d_1(x,y')=\min\nolimits_{n\in\mathbb N}[d_X(x,x_n)+C+d_X(y_n,y)]; 
$$ 
$$
d_2(x,y')=\min\nolimits_{n\in\mathbb N}[d_X(x,y_n)+C+d_X(x_n,y)]
$$
(it is clear that the minimum is attained on some $n\in\mathbb N$ as $x_n,y_n\to\infty$).
Let us show that $d_1$ is a metric on $X\sqcup X'$ (the case of $d_2$ is similar).

Due to symmetry, it suffices to check the two triangle inequalities for the triangle $xzy'$, $z\in X$: 
\begin{eqnarray*}
d_1(x,y')+d_1(z,y')&=&\min\nolimits_{n\in\mathbb N}[d_X(x,x_n)+C+d_X(y_n,y)]\\
&&+\ \min\nolimits_{m\in\mathbb N}[d_X(z,x_m)+C+d_X(y_m,y)]\\
&=&d_X(x,x_{n_x})+d_X(y_{n_x},y)+d_X(y,y_{n_z})+d_X(z,x_{n_z})+2C\\
&\geq&d_X(x,x_{n_x})+d_X(y_{n_x},y_{n_z})+d_X(z,x_{n_z})+2C\\
&\geq&d_X(x,x_{n_x})+(d_X(x_{n_x},x_{n_z})-C)+d_X(z,x_{n_z})+2C\\
&=&d_X(x,x_{n_x})+d_X(x_{n_x},x_{n_z}+d_X(z,x_{n_z})+C\\
&\geq&d_X(x,z)+C\ \geq\ d_X(x,z).
\end{eqnarray*}
and
\begin{eqnarray*}
d_1(x,y')&=&\min\nolimits_{n\in\mathbb N}[d_X(x,x_n)+C+d_X(y_n,y)]\\
&\leq&d_X(x,x_{n_z})+d_X(y_{n_z},y)+C\\
&\leq&d_X(x,z)+d_X(z,x_{n_z})+d_X(y_{n_z},y)+C\\
&=&d_X(x,z)+d_1(z,y').
\end{eqnarray*}

Let us evaluate $(d_2\circ d_1)(x_k,x'_k)$ and $(d_1\circ d_2)(x_k,x'_k)$. 

Taking fixed values $n=m=k$, $u=y_k$, we get
\begin{eqnarray*}
(d_2\circ d_1)(x_k,x'_k)&=&\inf\nolimits_{u\in X}\{\min\nolimits_{n\in\mathbb N}[d_X(x_k,x_n)+C+d_X(y_n,u)]\\
&&+\ \min\nolimits_{m\in\mathbb N}[d_X(u,y_m)+C+d_X(x_m,x_k)]\}\\
&\leq&\inf\nolimits_{u\in X}\{[d_X(x_k,x_k)+C+d_X(y_k,u)]\\
&&+\ [d_X(u,y_k)+C+d_X(x_k,x_k)]\}\\
&=&[d_X(x_k,x_k)+C]+[C+d_X(x_k,x_k)]\ =\ 2C.
\end{eqnarray*}

Using the triangle inequality for the triangle $x_nx_mu$ and (\ref{knm}), we get
\begin{eqnarray*}
(d_1\circ d_2)(x_k,x'_k)&=&\inf\nolimits_{u\in X}\{\min\nolimits_{n\in\mathbb N}[d_X(x_k,y_n)+C+d_X(x_n,u)]\\
&&+\ \min\nolimits_{m\in\mathbb N}[d_X(u,x_m)+C+d_X(y_m,x_k)]\}\\
&\geq&\inf\nolimits_{u\in X}\{\min\nolimits_{n\in\mathbb N}[d_X(x_k,y_n)+d_X(x_n,u)]\\
&&+\ \min\nolimits_{m\in\mathbb N}[d_X(u,x_m)+d_X(y_m,x_k)]\}\\
&\geq&
\min\nolimits_{n,m\in\mathbb N}[d_X(x_k,y_n)+d_X(x_n,x_m)+d_X(y_m,x_k)]\\
&>&k+d_X(x_n,x_m)+k\ >\ 2k.
\end{eqnarray*}

Thus, for the sequence $\{x_k\}_{k\in\mathbb N}$ of points in $X$, the distances $(d_2\circ d_1)(x_k,x'_k)$ are uniformly bounded, while $\lim_{k\to\infty}(d_1\circ d_2)(x_k,x'_k)=\infty$, hence the metrics $(d_2\circ d_1)$ and $d_1\circ d_2$ are not equivalent, i.e. $[d_2][d_1]\neq [d_1][d_2]$. 

\end{proof}


\begin{thebibliography}{9}


\bibitem{KLS}
A. Kar, J.-F. Lafont, B. Schmidt. {\it Rigidity of Almost-Isometric Universal Covers.}
Indiana Univ. Math. J.
{\bf 65} (2016), 585--613.

\bibitem{Lawson}
M. V. Lawson. {\it Inverse Semigroups: The Theory of Partial Symmetries.} World Scientific, 1998. 


\bibitem{Manuilov-Morphisms}
V. Manuilov. {\it Roe bimodules as morphisms of discrete metric spaces.} Russian J. Math. Phys., {\bf 26} (2019), 470--478.

\bibitem{M}
V. Manuilov. {\it Metrics on doubles as an inverse semigroup.} J. Geom. Anal., to appear.

\bibitem{M2}
V. Manuilov. {\it Metrics on doubles as an inverse semigroup II.} J. Math. Anal. Appl., to appear.



\bibitem{Paterson}
A. L. T. Paterson. {\it Groupoids, Inverse Semigroups, and their Operator Algebras.} Springer, 1998.







\end{thebibliography}
\end{document}